\documentclass[a4paper]{article}

\usepackage{microtype}
\usepackage{ellipsis}
\usepackage{graphicx}
\usepackage{amssymb}
\usepackage{amsmath}
\usepackage{amsthm}

\newtheorem{theorem}{Theorem}[section]

\newcommand{\im}{\operatorname{im}}
\newcommand{\Span}{\operatorname{span}}

\title{Reduced Basis Methods: Success, Limitations and Future Challenges
\thanks{This work has been supported by the German Federal Ministry of Education and Research (BMBF) under contract number 05M13PMA.}}

\author{Mario Ohlberger
	\thanks{Institute for Computational and Applied Mathematics \& Center for Nonlinear Science, University of M\"unster, Einsteinstr.\ 62, 48149 M\"unster, Germany, {\tt mario.ohlberger@uni-muenster.de}}
        \and Stephan Rave
	\thanks{Institute for Computational and Applied Mathematics, University of M\"unster, Einsteinstrasse 62, 48149 M\"unster, Germany, {\tt stephan.rave@uni-muenster.de}}}

\begin{document}

\maketitle

\begin{abstract}
Parametric model order reduction using reduced basis methods can be
an effective tool for obtaining quickly solvable reduced order models
of parametrized partial differential equation problems.
With speedups that can reach several orders of magnitude, reduced
basis methods enable high fidelity real-time simulations of complex
systems and dramatically reduce the computational costs in
many-query applications.
In this contribution we analyze the methodology,
mainly focussing on the theoretical aspects of the approach.
In particular we discuss what is known about the convergence
properties of these methods: when they succeed and when they are
bound to fail.
Moreover, we highlight some recent approaches employing nonlinear 
approximation techniques which aim to overcome the current limitations
of reduced basis methods.

\bigskip


\bigskip


\end{abstract}

\pagestyle{myheadings}
\thispagestyle{plain}
\markboth{M. OHLBERGER AND S. RAVE}{REDUCED BASIS METHODS}

\section{Introduction}

Over the last decade, reduced order modeling has become an integral part in many
numerical simulation workflows.
The reduced basis (RB) method is a popular choice for reducing the computational
complexity of parametrized partial differential equation (PDE) problems for either
real-time scenarios, where the solution of the problem needs to be known very
quickly under limited resources for a previously unknown parameter, or
multi-query scenarios, where the problem has to be solved repeatedly for many
different parameters.
The reduced order models obtained from RB methods during a computationally
intensive \emph{offline} phase typically involve approximation spaces of only
a few hundred or even less dimensions, leading to vast savings in computation
time when these models are solved during the so-called \emph{online} phase.

In this contribution, we first introduce the parametric model order
reduction problem in an abstract setting (Section~\ref{sec:abstract_problem}) and then give a short but
complete description of the RB method for the prototypic problem class of
linear, coercive, affinely decomposed problems, including a proof on the (sub-)exponential convergence 
of the approach (Section~\ref{sec:affine_coercive}).
Section~\ref{sec:extensions} contains
some pointers as to how the RB framework can be extended to other problem
classes.
In our presentation we will mostly focus on theoretical aspects of RB methods
and largely leave out any discussion of implementational issues and application
problems.
For more details on these aspects and RB methods in general, we refer to the
recent monographs~\cite{QuarteroniManzoniEtAl2016, HesthavenRozzaEtAl2016},
the tutorial~\cite{Ha14}, and the references therein.

RB and related methods can only succeed for problems which can be
approximated well using linear approximation spaces.
As we will see in Section~\ref{sec:limits}, this is typically not the case for advection driven
phenomena.
It is, therefore, inevitable to include nonlinear approximation techniques into
the RB framework to successfully handle this type of problem.
While this clearly poses a significant challenge for the methodology, first attempts
have been made towards this direction, some of which we will discuss in Section
5.2.

\section{Abstract problem formulation}
\label{sec:abstract_problem}
We are interested in solving parametric problems given by a
solution map 
\begin{equation*}
\Phi: \mathcal{P} \longrightarrow V
\end{equation*}
from a compact parameter
domain $\mathcal{P} \subset \mathbb{R}^P$ into some solution state space $V$,
which we will always assume to be a Hilbert space.
In all problems we consider, $\Phi(\mu)$ will be given as the solution of some
parametric partial differential equation.
Moreover, let $s: V \to \mathbb{R}^S$ be an $S$-dimensional output
functional which assigns to a state vector $v \in V$ the $S$ quantities of
interest $s(v)$. 
Note that the composition
\begin{equation*}
	s \circ \Phi: \mathcal{P} \longrightarrow V \longrightarrow
	\mathbb{R}^S,
\end{equation*}
which assigns to each parameter $\mu \in \mathcal{P}$ the quantities of interest
associated with the corresponding solution $\Phi(\mu)$, is a mapping between
low-dimensional spaces.
Assuming that both $\Phi$ and $s$ are sufficiently smooth, it is, therefore,
reasonable to assume that there exist quickly evaluable reduced order
models which offer a good approximation of $s \circ \Phi$.

Reduced basis methods are based on the idea of constructing state space
approximation spaces $V_N$ of low dimension $N$ for the so-called solution
manifold $\im(\Phi)$, and using the structure of the underlying equations defining
$\Phi$ to compute an approximation $\Phi_N(\mu) \in V_N$ of $\Phi(\mu)$.
By orthogonally projecting onto $V$, we can always assume that $V_N \subseteq V$
without diminishing the approximation quality.\footnote{This is not true for
	arbitrary Banach spaces $V$. E.g.\ consider the set of sequences in
	$c_0(\mathbb{N}) \subset l^\infty(\mathbb{N})$ which only assume the
	values 0 and 1. Each such function has $\|\cdot\|_\infty$ distance
	$1/2$ to the sequence with constant value $1/2$, but there is no
	finite-dimensional subspace of $c_0(\mathbb{N})$ with equal or lower
	best-approximation error.}
We then have the reduced approximation
\begin{equation*}
	s \circ \Phi_N: \mathcal{P} \longrightarrow V_N \longrightarrow
	\mathbb{R}^S
\end{equation*}
for the parameter-output mapping $s \circ \Phi$.

Given this abstract setup, the following questions, which will guide us through
the reminder of this article, are immediate:
\begin{enumerate}
\item Do there exists good approximation spaces $V_N$?
\item How to find a good approximation space $V_N$?
\item How to construct a quickly evaluable reduced solution map $\Phi_N$?
\item How to control the approximation errors $\Phi(\mu) - \Phi_N(\mu)$,
	$s(\Phi(\mu)) - s(\Phi_N(\mu))$?
\end{enumerate}
Assuming a positive answer to question 1, a multitude of answers have been given to
questions 2, 3 and 4 by now --- some of which we will discuss in the following sections ---
which yield more than satisfying results, both in theory and practice.
In particular, respecting the structure of the underlying equations defining
$\Phi$ allows for tight a posteriori estimates controlling the reduction error,
which in turn can be used to generate near-optimal approximation spaces $V_N$.
This is probably the greatest advantage over a straightforward interpolation of $s
\circ \Phi$, for which only crude error estimates exist and, especially for $P >
1$, the optimal selection of the interpolation points is unclear.

Moreover, we will see that, in fact, question 1 can be answered positively for large
classes of relevant problems (Section~\ref{sec:existence_of_approximation_spaces}).
Section~\ref{sec:limits} will be concerned with the case when no good linear
approximation spaces $V_N$ exist.

\section{An ideal world: coercive, affinely decomposed problems}
\label{sec:affine_coercive}

In this section we study the basic problem class of linear, coercive, affinely
decomposed problems, to which the reduced basis methodology is ideally fitted.
RB methods for other problem classes can be usually seen as extensions of the
ideas presented here.

We call a parametric problem linear, coercive if $\Phi(\mu)$ is given as the
solution $u_\mu$ of a variational problem
\begin{equation}\label{eq:coercive_affine_problem}
	a_\mu(u_\mu, v) = f(v) \qquad \forall v \in V,
\end{equation}
where, for each $\mu \in \mathcal{P}$, $a_\mu: V \times V \to \mathbb{R}$
is a continuous bilinear form on $V$ such that
$a_\mu(v, v) \geq C_{a_\mu} \|v\|^2$
with a strictly positive constant $C_{a_\mu}$, $f \in V^\prime$, and, in addition, the output
$s: V \to \mathbb{R}^S$ is a continuous linear map.
Continuity and coercivity of $a_\mu$ ensure the well-posedness of (\ref{eq:coercive_affine_problem}).
(A typical example would be, where $a_\mu$ is the variational form of an
elliptic partial differential operator on an appropriate Sobolev space and $f$
is the $L^2$-inner product with a given source term.)

We call the problem affinely decomposed if there are continuous mappings
$\theta_q: \mathcal{P} \to \mathbb{R}$ and continuous bilinear forms
$a_q: V \times V \to \mathbb{R}$ ($1 \leq q \leq Q$) such that
\begin{equation}\label{eq:affine_decomposition}
	a_\mu = \sum_{q=1}^{Q} \theta_q(\mu) a_q \qquad\forall \mu\in \mathcal{P}.
\end{equation}
Even though this assumption seems artificial at first sight, a large array of
real-world problems admit such an affine decomposition
(e.g.\ for diffusion equations, an affinely decomposed diffusivity tensor gives
rise to an affinely decomposed $a_\mu$).
In the following subsections we will give answers to the questions raised in
Section~\ref{sec:abstract_problem} for this class of problems.

\subsection{Existence of good approximation spaces}
\label{sec:existence_of_approximation_spaces}
The goal of RB methods is to find linear approximation spaces $V_N$ for which the
worst best-approximation error for elements of $\im(\Phi)$,
\begin{equation*}
d_{V_N}(\im(\Phi)):= \sup_{v \in \im(\Phi)}\inf_{v_N \in V_N} \|v - v_N\|,
\end{equation*}
is near the theoretical optimum
\begin{equation*}
d_N(\im(\Phi)) := \inf_{\substack{W \subseteq V \text{\ lin.\ subsp.} \\\dim W
		\leq N}}\,\sup_{v \in \im(\Phi)}\inf_{w \in W} \|v - w\|,
\end{equation*}
called the Kolmogorov $N$-width of $\im(\Phi)$.
Note that, since $V$ is a Hilbert space, the last infimum in both definitions
can be replaced by the norm of the defect of the orthogonal projection onto
$V_N$ (resp.\ $W$).
Moreover, an optimal
$N$-dimensional subspace $V_N$, for which $d_{V_N}(\im(\Phi)) = d_N(\im(\Phi))$,
always exists~\cite[Theorem II.2.3]{Pinkus1985}.
Nevertheless, the definition of the Kologorov $N$-width remains complex, and little is
known about the exact asymptotic behaviour of $d_N(\im(\Phi))$ in general.

For affinely decomposed problems, however, the $N$-widths always fall subexponentially fast
due to the holomorphy of the solution map $\Phi$.
While certainly known to experts, we believe a complete proof has never appeared
in the literature, so we provide it here:

\begin{theorem}
If $a_\mu$ is affinely decomposed according to (\ref{eq:affine_decomposition}),
then the Kolmogorov $N$-widths of the solution manifold of problem
(\ref{eq:coercive_affine_problem}) satisfy
\begin{equation*}
	d_N(\im(\Phi)) \leq C e^{-cN^{1/Q}},
\end{equation*}
with fixed constants $C, c > 0$.
\end{theorem}
\begin{proof}
Let $A_q: V \to V^\prime$ be the operators induced by
$a_q$, i.e.\ $A_q(v)[w]:=a_q(v, w)$.
By complex linearity, we extend these operators to continuous linear 
operators $A_q^\mathbb{C}: V^\mathbb{C} \to (V^\prime)^{\mathbb{C}} \cong
(V^\mathbb{C})^\prime$ between the complexification of $V$ and its dual.
Obviously, the (bilinear) mapping $\Psi: V^\mathbb{C} \times \mathbb{C}^Q \to
(V^\mathbb{C})^\prime$, $\Psi(v, c) := \sum_{q=1}^Q c_q A^\mathbb{C}_q(v)$
is holomorphic in the sense that it has a continuous, complex linear Fr\'echet derivative.
Moreover, $\partial_v\Psi(v, c) = \sum_{q=1}^Q c_q A^\mathbb{C}_q$ is, due to the coercivity of $a_\mu$, 
invertible for each $c \in \{(\theta_1(\mu), \ldots, \theta_Q(\mu)\ |\ \mu \in \mathcal{P}\} =: \hat{ \mathcal{P}}$.
Following~\cite{CohenDeVore2014}, we use the complex Banach space version of the implicit function 
theorem to deduce that $\hat{\Phi}: \hat{ \mathcal{P}} \to V^\mathbb{C}$,
$\hat{ \Phi}(\theta_1(\mu), \ldots, \theta_Q(\mu)) := \Phi(\mu)$ can be holomorphically
extended to an open neighbourhood $\hat{ \mathcal{P}} \subseteq \mathcal{O} \subseteq  \mathbb{C}^Q$. 

By compactness of $\hat{\mathcal{P}}$, there are finitely many $c_1, \ldots, c_M \in \hat{\mathcal{P}}$ and radii $r_1, \ldots, r_M$
such that $\hat{\mathcal{P}} \subset \bigcup_{m=1}^{M} D(c_m, r_m)$ and $\bigcup_{m=1}^M D(c_m, 2r_m) \subseteq \mathcal{O}$, where
$D(c, r) := \{z \in \mathbb{C}^Q\ |\ |z_q - c_q| < r,\, 1\leq q \leq Q\}$.
Holomorphy implies analyticity, thus there are for each $1 \leq m \leq M$ and
each multi-index $\alpha \in \mathbb{N}_0^Q$ vectors $v_{m, \alpha} \in V^\mathbb{C}$ such that $ \hat{\Phi}(z) = \sum_\alpha (z - c_m)^\alpha v_{m, \alpha}$,
converging absolutely for each $z \in D(c_m, 2r_m)$.
It is easy to see that, in fact, $v_{m, \alpha} \in V$.
Moreover, we have
\begin{equation*}
	C := \max_{1 \leq m \leq M} \sup_{z \in D(c_m, r_m)} \Bigl\|\,\sum_{\alpha} 2^\alpha(z - c_m)^\alpha v_{m, \alpha} \,\Bigr\| < \infty.
\end{equation*}
Note that there are $\frac{(Q+K)!}{Q!K!} \leq K^Q$ multi-indices $\alpha$ of length $Q$ and maximum degree $K$.
Let $K_N := \lfloor (M^{-1}N)^{1/Q} \rfloor$, and define
\begin{equation*}
	V_N := \Span\{v_{m, \alpha} \ |\ 1 \leq m \leq M, |\alpha| \leq K_N\}
	\subseteq V.
\end{equation*}

Now, for an arbitrary $\mu \in \mathcal{P}$ we can approximate $\Phi(\mu)= \hat{\Phi}(z)$, $z \in D(c_m, r_m)$,
by the truncated power series $\Phi_N(\mu) := \sum_{|\alpha| \leq K_N} \alpha(z - c_m)^\alpha v_{m, \alpha} \in V_N$.
We then obtain 
\begin{align*}
\|\Phi(\mu) - \Phi_N(\mu)\| &\leq \Bigl\|\,\sum_{|\alpha| \geq K_N + 1}
2^{-\alpha} \cdot 2^\alpha(z - c_m)^\alpha v_{m, \alpha}\, \Bigr\| \\
                            & \leq C 2^{-(K_N + 1)} \leq C e^{- \operatorname{ln}(2)M^{-1/Q}N^{1/Q}}.
\end{align*}
\end{proof}

Note, that such type of estimate will degenerate for $Q \to \infty$.
On the other hand, we can replace $Q$ by $P$ whenever the parameter functionals
$\theta_q$ are analytic.
In fact, we needed the affine decomposition $(\ref{eq:affine_decomposition})$ of
$a_\mu$ only to establish the analyticity of $\Phi$.
The implicit function theorem argument from~\cite{CohenDeVore2014} can be
applied to various other problem classes, for which, therefore, the same type of
result holds.

Algebraic convergence rates for infinite affine decompositions where the coefficients
satisfy some summability condition are shown in \cite{CohenDevoreEtAl2011}. 

\subsection{Definition of $\Phi_N$}
Assuming a reduced subspace $V_N$ has already been constructed, we determine the
RB solution $\Phi_N(\mu):=u_{N, \mu} \in V_N$ via Galerkin projection of the original
equation as the solution of
\begin{equation}\label{eq:reduced_problem}
	a_\mu(u_{N, \mu}, v_N) = f(v_N) \qquad\forall v_N\in V.
\end{equation}
As usual, C\'ea's Lemma gives use the following quasi-optimality estimate for the
model reduction error:
\begin{equation}\label{eq:cea_lemma}
	\|\Phi(\mu) - \Phi_N(\mu)\| \leq \frac{\|a_\mu\|}{C_{a_\mu}}
	  \inf_{v_N \in V_N} \|\Phi(\mu) - v_N\|.
\end{equation}
Note that if we pre-compute the matrices of the bilinear forms $a_q$ and the
coefficients of $f$ and $s$ w.r.t.\ to a basis of $V_N$,
computing $s \circ \Phi_N(\mu)$ will require only $\mathcal{O}(QN^2)$ operations
for system matrix assembly, $\mathcal{O}(N^3)$ operations for the solution of
the reduced system and $\mathcal{O}(SN)$ operations for the evaluation of the
output during the online phase.
No operations involving the space $V$ need to be performed.

\subsection{Error control}
We use a standard residual-based error estimator to bound the model reduction
error.
Let the reduced residual be given by $\mathcal{R}_\mu(u_{N, \mu})[w]:= f(w) -
a_\mu(u_{N, \mu}, w)$ for  $w \in V$.
The well-known residual-error relation
$\mathcal{R}_\mu(u_{N, \mu})[v] = a_\mu(u_\mu - u_{N, \mu}, v)$
yields together with the coercivity of $a_\mu$:
\begin{equation}\label{eq:error_estimator}
\|\Phi(\mu) - \Phi_N(\mu)\| \leq \frac{1}{C_{a_\mu}}
\|\mathcal{R}_\mu(u_{N,\mu})\|_{V^\prime} \leq \frac{\|a_\mu\|}{C_{a_\mu}} \|\Phi(\mu) 
-\Phi_N(\mu)\|.
\end{equation}
Thus, $1/C_{a_\mu}\cdot \|\mathcal{R}_\mu(u_{N,\mu})\|_{V^\prime}$ is a guaranteed upper
bound for the model reduction error with effectivity $\|a_\mu\| / C_{a_\mu}$.
An upper bound for the output error
is then given by
\begin{equation*}
\|s\circ\Phi(\mu) - s\circ\Phi_N(\mu)\| \leq \frac{\|s\|}{C_{a_\mu}}
\|\mathcal{R}_\mu(u_{N,\mu})\|_{V^\prime}.
\end{equation*}
This upper bound and the output approximation itself can be further improved using a
\emph{primal-dual} approximation approach (e.g.~\cite{QuarteroniRozzaEtAl2011}).

Note that since $V^\prime$ is a Hilbert space, we have
\begin{equation*}
\|\mathcal{R}_\mu(\ast)\|^2 = (f - a_\mu(\ast, \cdot),\ f - a_\mu(\ast,
\cdot))_{V^\prime}.
\end{equation*}
Pre-computing all appearing scalar products w.r.t.\ the affine decomposition of
$a_\mu$ and a basis of $V_N$,
this residual norm can be evaluated efficiently online with $\mathcal{O}(Q^2N^2)$
operations.
Again, no operations involving the space $V$ are required.

For the complete evaluation of (\ref{eq:error_estimator}), the coercivity constant
$C_{a_\mu}$, or a lower bound for it, must be known.
In many cases, good lower bounds for the problem at hand are known a priori.
If not, the \emph{successive constraint method} \cite{HuynhRozzaEtAl2007} is an
well-established approach to compute such lower bounds online for arbitrary
$\mu\in\mathcal{P}$ using offline pre-computed lower bounds for
$C_{a_{\mu_i}}$ for certain well-chosen $\mu_i$.

\subsection{Construction of $V_N$}
A natural approach for the construction of approximation spaces $V_N$ for
$\im(\Phi)$ during the offline phase is to iteratively enlarge the reduced space by
an element of $\im(\Phi)$ which maximizes the best-approximation error for
the current reduced space. 
Such \emph{greedy} algorithms have been studied extensively in approximation
theory.
While it is clear that greedy algorithms will not produce best-approximating
spaces for the solution manifold\footnote{E.g., let $\mathcal{M}:=\{[1~0], [0~1]\} 
	\subset \mathbb{R}^2$. Then $d_{V_1}(\mathcal{M}) = 1$ for a $V_1$ generated by
	a greedy algorithm, whereas $d_1(\mathcal{M}) = 1/\sqrt{2}$.},
several quasi best-approximation results have been derived in the literature.
For their analysis in the context of RB methods see
\cite{BinevCohenEtAl2011, BuffaMadayEtAl2012, DeVorePetrovaEtAl2013}.
In particular, the following has been shown in \cite{DeVorePetrovaEtAl2013}:
We call $u_1, \ldots, u_N \in \im(\Phi)$ a \emph{weak greedy} sequence for $\im(\Phi)$
if there is a $\gamma > 0$ s.t.
\begin{equation*}
	\sup_{v \in V_{n-1}} \|u_n - v\| \geq \gamma\cdot d_{V_{n-1}}(\im(\Phi)),
	\qquad V_n := \Span\{v_1, \ldots V_{n}\}, \qquad 1\leq n \leq N,
\end{equation*}
with $V_0 := \{0\}$.
Now, if $d_N(\im(\Phi)) \leq Ce^{-cN^\alpha}$ for all N and the spaces $V_N$
have been constructed from a weak greedy sequence with parameter $\gamma$, then
$d_{V_N}(\im(\Phi)) \leq \sqrt{2C}\gamma^{-1}e^{-c^\prime N^\alpha}$,
where $c^\prime = 2^{-1-2\alpha}c$. Similar results have been obtained for
algebraic convergence of $d_n(\im(\Phi))$.

A weak greedy sequence for $\im(\Phi)$ can be constructed using the error
estimator (\ref{eq:error_estimator}) as a surrogate for the best-approximation
error in $V_N$: in each iteration, we extend the reduced space by a
$\Phi(\mu)$ where $\mu$ is a maximizer of the estimated model reduction error.
Due to the effectivity estimate (\ref{eq:error_estimator}) and C\'ea's Lemma
(\ref{eq:cea_lemma}), one can easily see that this, indeed, yields a weak greedy
sequence with parameter $\gamma = \inf_{\mu \in \mathcal{P}}(\|a_\mu\|/C_{a_\mu})^{-2}$.

\subsection{Implementation and the notion of truth}
While everything we have discussed so for applies to arbitrary (possibly) infinite
dimensional Hilbert spaces $V$, the actual implementation of the RB method
will only be possible when $V$ is finite dimensional.
In practice, the original analytical problem, posed on some infinite function space $V$, is
therefore replaced by a discrete approximation that is so highly resolved that
the discretization error is negligible w.r.t.\ the model reduction error.
In the literature, this approximation is often referred to as the \emph{truth}
approximation.

Typically, computing the truth approximation for a single parameter $\mu$ will
be computationally expensive (which is why model reduction is desired).
However, such computations only need to be performed in the offline
phase of the scheme and only to compute basis vectors (and the associated reduced
model) for $V_N$.
In particular, thanks to the usage of the online-efficient error estimator
(\ref{eq:error_estimator}) to select the next parameter for the extension of
$V_N$, no high-dimensional operations are needed to find this
parameter.
Note that the typically infinite parameter space $\mathcal{P}$
will still have to be replaced by a finite training set $\mathcal{S}_{train} \subseteq
\mathcal{P}$ to make the search for this parameter feasible.
However, as the error estimator can be evaluated very quickly, very large training
sets that densely sample $\mathcal{P}$ are tractable. 
Moreover, adaptive algorithms are available (e.g.~\cite{HaasdonkDihlmannEtAl2011}), to
refine $\mathcal{S}_{train}$ only where needed.

Recently, new approaches~\cite{Yano2014, OhlbergerSchindler2015, AliSteihEtAl2014}
have appeared which provide online
efficient estimators that measure the model reduction error w.r.t.\ the
analytical solution of the given problem. 
Such approaches not only allow to certify that the reduced solution has a
guaranteed approximation quality w.r.t.\ the PDE model one is interested in,
but also enable adaptive methods for the on-demand refinement of the truth approximation.

\section{Extensions}
\label{sec:extensions}
We have seen in the previous section that for linear, coercive, affinely
decomposed problems, the RB approach yields low-dimensional, quickly solvable
reduced order models with (sub-)exponentially fast decaying error, which can be
rigorously bound using an efficient a posteriori error estimator.
Based on these fundamental ideas, extensions of the methodology to a wide array
of problem classes have been proposed.
We can only mention a few important ideas.

\subsection{Time-dependent problems}
In the method of lines approach, (time-dependent) parabolic partial differential
equations are first approximated by replacing the space differential operators
of the equation by an appropriate discrete approximation, yielding an ordinary
differential equation system in time, which is then solved using standard ODE
time stepping methods.
The same approach can be applied in the RB setting.
Thus, we search for reduced spaces $V_N$ which approximate the solution
trajectories of the given problem for each parameter $\mu$ and point in time
$t$. 
I.e.\ $d_{V_N}(\mathcal{M}_\Phi^t)$, where $\mathcal{M}_\Phi^t:=\{\Phi(\mu)[t]\
|\ \mu \in \mathcal{P}, t \in [0, T]\}$, should be as small as possible.

Since errors propagate through time, it is easy to conceive that greedily
selecting a $\Phi(\mu)[t]$ which maximizes the model reduction error will not
yield good results.
A better approach is to first select a maximum error trajectory $\Phi(\mu^*)$
and then add the first modes of a proper orthogonal decomposition (POD) of the
projection error of $\Phi(\mu^*)$ onto $V_N$ to $V_N$ (\textsc{POD-Greedy},
\cite{HaasdonkOhlberger2008})\footnote{I.e., one computes the truncated
	singular value decomposition of the linear mapping $\mathbb{R}^N \to V$,
	$n \mapsto (I - P_{V_N})\Phi(\mu^*)[t_n]$, where $t_0, \ldots, t_N$ is some
time discretization of $[0, T]$.}.
In~\cite{Haasdonk2013} it was shown that similar to the stationary case, the
\textsc{POD-Greedy} algorithm yields quasi-optimal convergence rates, e.g.\ in
the sense that (sub-)exponential decay of the $N$-widths of $\mathcal{M}_\Phi^t$ carries
over to the decay of $d_{V_N}(\mathcal{M}_\Phi^t)$.
As in the classical finite element setting, a posteriori error estimators for
the reduction error can be obtained by time integration of the error-residual
relation.

These error estimators, however, show bad long-time effectivity, in particular
for singularly perturbed or non-coercive (see below) problems. 
To mitigate this problems, space-time variational formulations for the reduced
order model, which allow for tighter error bounds, have been considered (e.g.\ 
\cite{UrbanPatera2012, Yano2014a}).

\subsection{Inf-sup stable problems}
A crucial prerequisite for the applicability of Galerkin projection-based model
order reduction is a manageable condition of the problem. I.e.\ the quotient
$\kappa_\mu := \|a_\mu\| / C_{a_\mu}$ has to be of modest size,
as it determines the quality of the reduced solution (\ref{eq:cea_lemma}).
While $\kappa_\mu$ has no significant effect on the asymptotic behaviour of RB
methods, a too large $\kappa_\mu$ can render the RB approach practically
infeasible.

Typical examples include advection diffusion equations with small diffusivity
or, as the limit case, hyperbolic equations where coercivity is completely lost.
Many of these problems are still inf-sup stable, i.e.
\begin{equation*}
\inf_{0 \neq v \in V}\sup_{0 \neq w \in V} a_\mu(v, w) / \|v\|\|w\| > 0.
\end{equation*}
Assuming that $a_\mu$ is inf-sup stable on $V_N$, a similar quasi-optimality
result to C\'ea's Lemma (\ref{eq:cea_lemma}) holds\footnote{For infinite
	dimensional $V$ we additionally need to assume non-degeneracy of $a_\mu$
	in the second variable.}.
However, contrary to coercivity, inf-sup stability is not inherited by arbitrary
subspaces $V_N$.
Petrov-Galerkin formulations, where appropriate test spaces for the reduced
variational problem (\ref{eq:reduced_problem}) are constructed, are a natural
setting to preserve the inf-sup stability of the reduced bilinear form.
Several approaches have by now appeared in the literature.
We specifically mention~\cite{DahmenPleskenEtAl2014} where, in addition,
problem adapted norms on the trial and test spaces are chosen to ensure optimal
stability of the reduced problem.
In the recent work~\cite{ZahmNouy}, stability of the reduced problem is improved
using preconditioners obtained from an online efficient interpolation scheme.

\subsection{Not affinely decomposed and nonlinear problems}
Crucial for being able to quickly evaluate $\Phi_N$ is the affine decomposition
of $a_\mu$ (\ref{eq:affine_decomposition}) which allows us to assemble the
system matrix for (\ref{eq:reduced_problem}) by linear combination of the
pre-computed, non-parametric reduced matrices of $a_q$.
For problems where such an affine decomposition is not given,
a widely adopted approach is to approximate $a_\mu$ by some $\hat{a}_\mu$
admitting an affine decomposition. $\hat{a}_\mu$ is determined using an interpolation scheme,
where the interpolation points and interpolation basis are constructed from
snapshot data for the parametric object to interpolate.
Originally, this \emph{empirical interpolation} method was introduced for parametric
data functions (appearing in the definition of
$a_\mu$)~\cite{BarraultMadayEtAl2004},
and has since then been extended to general, possibly nonlinear, 
operators~\cite{HaasdonkOhlbergerEtAl2008, ChaturantabutSorensen2010,
CarlbergBou-MoslehEtAl2011, DrohmannHaasdonkEtAl2012}.

\section{Limits of reduced basis methods}\label{sec:limits}
By now, the reduced basis methodology has matured to a point where a large body
of problems admitting rapidly decaying Kolmogorov $N$-widths can be
handled with great success.
However, many relevant problems, in particular advection dominated problems,
suffer from a very slow decay of the $N$-widths, even though the structure of
their solutions suggests that efficient reduced order models should exist.
In this section we will give a very simple example which falls into this
category of problems and briefly discuss first attempts that have been made to
tackle these problems by means of nonlinear approximation techniques.

\subsection{The need for nonlinear approximation}
A slow decay of the Kolmogorov $N$-widths can already be observed for simple
linear advection problems involving jump discontinuities:
\begin{equation}\label{eq:severe_problem}
\begin{gathered}
	\partial_t u_\mu(x, t) + \mu \cdot \partial_x u_\mu(x, t) = 0 \qquad \mu, x, t \in [0, 1]\\
	u_\mu(x, 0)  = 0,\quad u_\mu(0, t) = 1.
\end{gathered}
\end{equation}
If we choose a method of lines approach, even a single solution trajectory of
(\ref{eq:severe_problem}) cannot be well-approximated using linear spaces.
I.e.\ consider $\mathcal{M}:=\{u_1(t)\ |\ t \in [0, 1]\} \subset L^2([0,1])$.
One readily checks that for each $N \in \mathbb{N}$, $\mathcal{M}$ contains the pairwise
orthogonal functions $\psi_{N,n}$, $1 \leq n \leq N$, of norm $N^{-1/2}$ given by
\begin{equation*}
\psi_{N,n}(x):=
\begin{cases}
	1 & \frac{n-1}{N} \leq x \leq \frac{n}{N}  \\
	0 & \text{otherwise}
\end{cases}.
\end{equation*}
Thus,
\begin{align*}
	d_{N}(\mathcal{M}) &\geq d_N\bigl(\{\psi_{2N, n}\ |\ 1 \leq n \leq
	2N\}\bigr)\\ &= (2N)^{-1/2} \cdot d_N\bigl(\{(2N)^{1/2}\psi_{2N, n}\ |\ 1 \leq n \leq 2N\}\bigr).
\end{align*}
Note that the latter set can be isometrically mapped to the canonical
orthonormal basis in $\mathbb{R}^{2N}$.
Since, by definition, the Kolmogorov $N$-width is invariant under taking the
balanced convex hull, we obtain using Corollary~IV.2.11 of \cite{Pinkus1985}
\begin{align*}
	d_{N}(\mathcal{M}) &\geq (2N)^{-1/2} \cdot d_N\bigl( \{x \in
	\mathbb{R}^{2N}\ |\ \|x\|_1 \leq 1\} \bigr)\\
	&= (2N)^{-1/2} \cdot
	\left(\frac{2N - N}{2N}\right)^{1/2} = \frac{1}{2}N^{-1/2}.
\end{align*}

Note that this convergence issue is not due to the methods of line approach.
Even if we switch to a space-time formulation, treating 
(\ref{eq:severe_problem}) as a stationary equation on
$[0,1]^2$, will not solve the problem in the parametric case.
Using the same arguments, one easily sees that, still, $d_N(\{u_\mu\ |\ \mu \in [0, 1]\})
\sim N^{-1/2}$.

No matter what, classical RB methods or any other model reduction approach for
which $\Phi_N$ maps to a linear subspace $V_N \subseteq V$ are bound to fail for this type
of problem.
Only methods for which $V_N$ is a nonlinear subspace of $V$ can be successful.

Regarding application problems where the described behaviour is observed, we
specifically mention the challenging class of kinetic transport equations,
for which first model reduction attempts are presented in \cite{BrunkenOhlbergerSmetana2016}
 and in the references therein.

\subsection{First attempts}
By now, several attempts have been made to extend the RB methodology
towards nonlinear approximation.
In the following, we will briefly discuss the most important approaches we are aware of.
Most of these approaches are still in their early stages, usually only tested
for selected model problems and with little theoretical underpinning.
Nevertheless, promising first steps have been taken, and in view of the variety of
the approaches, it seems likely that substantial progress on such methods can be made
in the years to follow.

\paragraph{Dictionary-based approximation}
An obvious generalization of linear approximation in a single space $V_N$ is to
employ a dictionary $\mathcal{D}$ of linear reduced spaces from which an
appropriate $V_N \in \mathcal{D}$ is selected depending on the parameter $\mu$
or point in time for which the solution is to be approximated
\cite{HaasdonkDihlmannEtAl2011, DihlmannDrohmannEtAl2011, KaulmannHaasdonk2013}.
However, while such approaches may increase online efficiency by allowing
smaller approximation spaces, the overall number of required basis vectors is
still controlled by the Kolmogorov $N$-width:
\begin{equation*}
	\sup_{\mu \in \mathcal{P}} \min_{V_N \in \mathcal{D}} \inf_{v \in V}
	\|\Phi(\mu) - v\| \geq d_{\Span(\bigcup_{V_N
			\in \mathcal{D}} V_N)}(\im(\Phi)) 
			\geq d_{\sum_{V_N \in \mathcal{D}} \dim V_N}
			(\im(\Phi)).
\end{equation*}
Thus, to achieve an error of $\varepsilon$ for the approximation of
(\ref{eq:severe_problem}), still a total amount of $\varepsilon^{-2}$ basis
vectors has to be included in $\mathcal{D}$.
While such an approach might be feasible in one space dimension, where all
possible locations of the discontinuity can already be obtained from one
solution trajectory, offline computations in higher space dimensions will be
prohibitively expensive.

In~\cite{Carlberg2015}, an \emph{adaptive h-refinement} technique for RB
spaces is presented.
Starting from a coarse reduced basis obtained from global solution
snapshots, a hierarchy of approximation spaces can be adaptively generated on-the-fly
by dissecting the basis vectors w.r.t.\ a pre-computed hierarchy of DOF set partitions.
This approach mitigates the need for large numbers of solution snapshots 
while allowing arbitrarily accurate approximation spaces, albeit at an
increased computational effort online.

\paragraph{Shock detection}
Another approach, which is geared specifically towards treating moving
discontinuities is to detect the space-time regions with shocks or low
regularity and use low-dimensional linear approximation spaces only outside
these regions. 
In~\cite{ConstantineIaccarino2012}, first an interpolation method in parameter
space is used to obtain a reduced solution. The Jacobian of the interpolant is
then used to detect non-smooth space-time regions in which then a finely
resolved correction is computed.

In~\cite{TaddeiPerottoEtAl2015}, a more elaborate shock capturing algorithm is
developed to obtain an online efficient approximation of the trajectory $x_s(t)$
of the discontinuity location over time.
This information is then used to transform the space-time domain into three
parts (before discontinuity appears, left and right of discontinuity)
which are transformed to reference domains.
On these reference domains, empirical interpolation is finally used to
obtain a low-order approximation of the smooth solution components.
Since the values of the transformed solution components need only to be
known at the given space-time interpolation points, these values can be quickly
computed using the methods of characteristics.  

To our knowledge, these methods have not been successfully applied in higher space
dimensions yet.

\paragraph{Nonlinear parametrization}
A more generic approximation approach is to describe nonlinear approximation
spaces $V_N$ by a nonlinear parametrization.
For advection driven problems, a natural choice is to incorporate
transformations of the underlying spatial domain (shifts, rotations or more
general transformations) into the parametrization.

In~\cite{OhlbergerRave2013}, these transformations are assumed to be given
by a Lie group $G$ of mappings acting on $V$.
The reduced solution manifold $V_N$ is then given by all vectors $g.v$ where 
$g \in G$ accounts for the dynamics of the solution and $v \in
\hat{V}_N$ describes the (ideally) stationary shape of the solution.
This \emph{ansatz} is then substituted into the given differential equation,
and the algebraic constraint that the evolution of $v(t)$ should be orthogonal
to the action of the Lie algebra of $G$ at $v(t)$, is added to determine
the additional degrees of freedom.
Given the invariance of the problem under the action of $G$, standard RB
techniques for approximating $v(t) \in \hat{V}_N$ yield an online efficient
reduced order approximation of the resulting frozen equation system.

In~\cite{Welper2015}, a parameter space interpolation scheme is developed where
$u_\mu \in V$ is approximated by an expression of the form
$$u_{\mu, N}(x) = \sum_{\eta \in \mathcal{P}_N} l_\eta(\mu)u_\eta(\phi(\mu,
\eta)(x)),$$ where $l_\eta$ are Lagrange interpolation polynomials associated with
the interpolation points $\eta$ and $u_\eta \in V$ are solutions snapshots which
are transformed via a mapping $\varphi: \mathcal{P} \times \mathcal{P} \times
\Omega \to \Omega$.
An optimization algorithm w.r.t.\ a training set of solution snapshots is
used to determine $\phi$ during the offline phase. 

Low-order approximations of advection dominated trajectories of the
form $$u(x, t) \approx [u_0(Y(x,t)) + R(Y(x,t),t)]\det(\nabla_x Y(x,t))$$ are considered in
\cite{IolloLombardi2014}.
While standard POD is used to approximate the residual part $R(x, t)$ of the
trajectory, the transformation $Y$ is approximated by a principal component
analysis based on the Wasserstein distances between the snapshots $u(x, t_i)$,
with modes being obtained by solving Monge-Kantorovich optimal transport problems
w.r.t.\ the reference mode $u_0(x)$.

\paragraph{Approximation based on Lax pairs}
Finally, we mention a new model reduction approach based on the use of so called
Lax pairs in the theory of integrable systems~\cite{GerbeauLombardi2014}.
Given a solution trajectory $u(t)$ of an evolution equation, the associated Schr\"odinger
operator with potential $-\chi u$ at time $t$ is given by $\mathcal{L}_{\chi u}(t)\varphi
= - \Delta \varphi - \chi u(t) \varphi$.
With $\lambda_m(t), \varphi_m(t)$ denoting the $m$-th eigenvalue (eigenvector) of
$\mathcal{L}_{\chi u}(t)$, there are operators $\mathcal{M}(t)$ such that
$\partial_t\varphi_m(t) = \mathcal{M}(t)\varphi_m(t)$.
One then has
\begin{equation}\label{eq:approximate_lax_pair}
(\mathcal{L}_{\chi u}(t) + [\mathcal{L}_{\chi u}(t), \mathcal{M}(t)])
\varphi_m(t) = \partial_t \lambda_m(t) \varphi_m (t),
\end{equation}
where $[\mathcal{A}, \mathcal{B}] = \mathcal{A}\mathcal{B} -
\mathcal{B}\mathcal{A}$.\footnote{$(\mathcal{L}_{\chi u}, \mathcal{M})$ is called a Lax pair when the
right-hand side of (\ref{eq:approximate_lax_pair}) vanishes for all $m$, i.e.\
the eigenvalues of $\mathcal{L}_{\chi u}$ are constant.}
Using the $\varphi_m$ as a moving coordinate frame which is truncated to the
first $N$ eigenvectors, the authors deduce from (\ref{eq:approximate_lax_pair}) a reduced ordinary
differential equation system which describes the evolution of the coordinates of
the reduced approximation of $u(t)$ w.r.t.\ this coordinate frame.

\bibliographystyle{siam}
\bibliography{bib}
 
\end{document}